\documentclass{article}
\usepackage[cmex10]{amsmath}
\usepackage{amsfonts} 
\usepackage{graphicx}
\usepackage{amsthm}

\newtheorem{theorem}{Theorem}[section]
\newtheorem{definition}[theorem]{Definition}
\newtheorem{lemma}[theorem]{Lemma}

\newcommand{\Keywords}[1]{\par\noindent
{\small{\bf Keywords\/}: #1}}

\newtheorem*{lema1}{Lemma \ref{lema}}
\newtheorem*{teorema1}{Theorem \ref{propqvi}}

\begin{document}

\title{Optimal Maintenance Policy for a Compound Poisson Shock Model}

\author{Mauricio Junca\thanks{Department
of Mathematics, Universidad de los Andes, Bogot\'a,
Colombia, e-mail: mj.junca20@uniandes.edu.co.} \and Mauricio Sanchez-Silva\thanks{Department
of Civil and Environmental Engineering, Universidad de los Andes, Bogot\'a,
Colombia, e-mail: msanchez@uniandes.edu.co.}
}

\maketitle

\begin{abstract}
Engineered and infrastructure systems deteriorate (e.g., loss capacity) as a result of adverse environmental or external conditions. Modeling deterioration is essential to define optimum design strategies and inspection and maintenance (intervention) programs. In particular, the main purpose of maintenance is to increase the system availability by extending the life of the system. Most strategies for maintenance optimization focus on defining long term strategies based on the system's condition at the decision time (e.g., $t=0$). However, due to the large uncertainty in the system's performance through life, an optimal maintenance policy requires both permanent monitoring and a cost-efficient plan of interventions.  This paper presents a model to define an optimal maintenance policy of systems that deteriorate as a result of shocks. Deterioration caused by shocks is modeled as a compound Poisson process and the optimal maintenance strategy is based on an impulse control model.  In the model the optimal time and size of interventions are executed according the the system state, which is obtained from permanent monitoring.
\newline
\end{abstract}

\section{Introduction}

Engineered and  infrastructure systems deteriorate as a result of the normal use or due to external demands imposed by adverse environmental conditions. The main challenge in modeling deterioration is to manage the damage accumulation mechanisms and the associated uncertainties. Deterioration mechanisms can be divided into progressive (e.g., corrosion, fatigue) and shock-based (e.g., earthquakes, blasts)\cite{San11}. 

In the particular case of large infrastructure, progressive deterioration can be caused by, for instance, chloride ingress, which leads usually to steel corrosion, loss of effective cross-section of steel reinforcement in RC structures, concrete cracking, loss of bond and spalling \cite{Bastidas09}. The details of these deterioration mechanisms are beyond the scope of the paper but are well described by, for instance, \cite{Klutke02} and \cite{Frang04}. On the other hand, deterioration caused by extreme events is usually associated with earthquakes, hurricanes or blasts (including both accidents and terrorists attacks). Extensive research has been carried out on mathematical models for shock degradation in infrastructure and in other types of engineered artifacts; for more details see \cite{BarlowPosh65}, \cite{AvenJan99}, \cite{Naka76}, \cite{Naka85}, \cite{Feldman77a}, \cite{San11}, \cite{YeCX} and \cite{WangPham}.

A {\it maintenance program} is as a set of actions directed to keep a deteriorating system (e.g., machine, building, infrastructure) operating above a pre-specified level of service; thus, maintenance is carried out to improve the availability or to extend the life of the system \cite{Wortman94} \cite{Naka05}. The long-term benefits of an optimum maintenance strategy include: improvement of the system reliability, replacement cost reduction, system downtime decrease and better spares inventory management\cite{Naka05}.

Frequently, a comprehensive maintenance program includes \textit{preventive} and/or \textit{corrective} or \textit{reactive} actions \cite{Misra92}, \cite{Dekker96}. Preventive maintenance involves all actions directed to avoid failure or to avoid higher cost at a later stage by keeping the component in a safe or operational condition. Preventive maintenance is frequently carried out without knowing the actual state of the component at the time of the intervention. On the other hand, corrective maintenance focuses on interventions once failure has been identified. Corrective maintenance is frequently more expensive than preventive maintenance since the cost may include, in addition to the repair cost, downtime costs or replacement of undamaged system components. While preventive maintenance is commonly carried out at fixed time intervals, corrective maintenance is performed at unpredictable intervals because failure times cannot be known {\sl a priori} \cite{AValdez89}\cite{Wang06}.

Most work on defining maintenance strategies focuses on establishing \textit{action plans} of interventions \cite{AValdez89}. However, for infrastructure systems or large systems with long expected lifetimes (e.g., bridges that last over 75 years), these approaches are not realistic. Over time, the system functionality may be modified, some unplanned demands may appear or technology changes, thus, forcing changes in the original maintenance plans. Predefined long-term maintenance plans are rarely implemented or they need to be modified as new information becomes available. 

Given the uncertainty associated to the component's performance, the best maintenance policy should be based on a permanent monitoring strategy that leads to optimum interventions. This paper presents a maintenance strategy based on \textit{impulse control} models in which the time at which maintenance is carried out and the extent of interventions are optimized simultaneously to maximize the cost-benefit relationship. In the model the optimal time and size of interventions are executed according the the system state, which is obtained from permanent monitoring. 

The paper is organized as follows: the basic concepts of the impulse control model are presented in section II and, in section III, some basic theorems are outlined. The core of the proposed optimum strategy is presented and discussed in section IV and a description of the numerical solution is presented in section V. Finally, in section VI the proposed approach is illustrated with an example and the main results and conclusions are summarized in section VII.

\section{Impulse Control Model}

Let's define a system component (e.g., bridge) whose performance is defined by a random variable $R$; for instance, it can be the component's reliability. Furthermore, assume that the component is subject to shocks, which occur according to a Poisson process, and every shock causes a random amount of damage $S$. Thus, if $(\Omega,\mathcal{F},P)$ is the probability space in which we define all the stochastic quantities, we define the process $R=\{R_t,{t\geq0}\}$

\begin{equation}
R_t=r_0-\sum\limits_{i=1}^{N_t}S_i
\end{equation}

\noindent where $N=\{N_t,{t\geq0}\}$ is an homogeneous Poisson process with intensity $\lambda>0$. The sequence of the sizes of the shocks $\{S_i\}_{i\in\mathbb{N}}$ are independent and identically distributed random variables with probability distribution $F$ on $(0,\infty)$. We assume that $\{S_i\}_{i\in\mathbb{N}}$ is independent of the Poisson process $N$. The initial reliability level is $R_{0-}=r_0$ (Fig. \ref{Fig1}a).

\begin{definition}
A \textit{maintenance policy} for the system is a double sequence $\nu=\{(\tau_i,\zeta_i)\}_{i\in\mathbb{N}}$  of intervention times $\tau_i$ at which the performance is improved an amount $\zeta_i$ (in $R$-units). The policy is an \textit{impulse control} if satisfies the following conditions:

\begin{enumerate}
\item $0\leq\tau_i\leq\tau_{i+1}$ for all $i\in\mathbb{N}$,
\item $\tau_i$ is a stopping time with respect to the filtration $\mathcal{F}_t=\sigma\{R_{s-}|s\leq t\}$ for $t\geq0$,
\item $\zeta_i$ is a $\mathcal{F}_{\tau_i}$-measurable random variable,
\end{enumerate}
\end{definition}

Note that the class of impulse control policies is very general and include, in particular, policies with fixed time interventions.

Given an impulse control $\nu$, the controlled process $R^{\nu}=\{R^{\nu}_t,{t\geq0}\}$ is defined by
\begin{equation}
R_t^{\nu}=r_0-\sum\limits_{i=1}^{N_t}S_i+\sum\limits_{i=1}^{\infty}\zeta_iI_{\{\tau_i\leq t\}}.
\end{equation}
where $I_{\{\tau_i\leq t\}}$ is the indicator function  (Fig. \ref{Fig1}b).

\begin{figure}[htbp] 
\begin{center}
\includegraphics[scale=0.575]{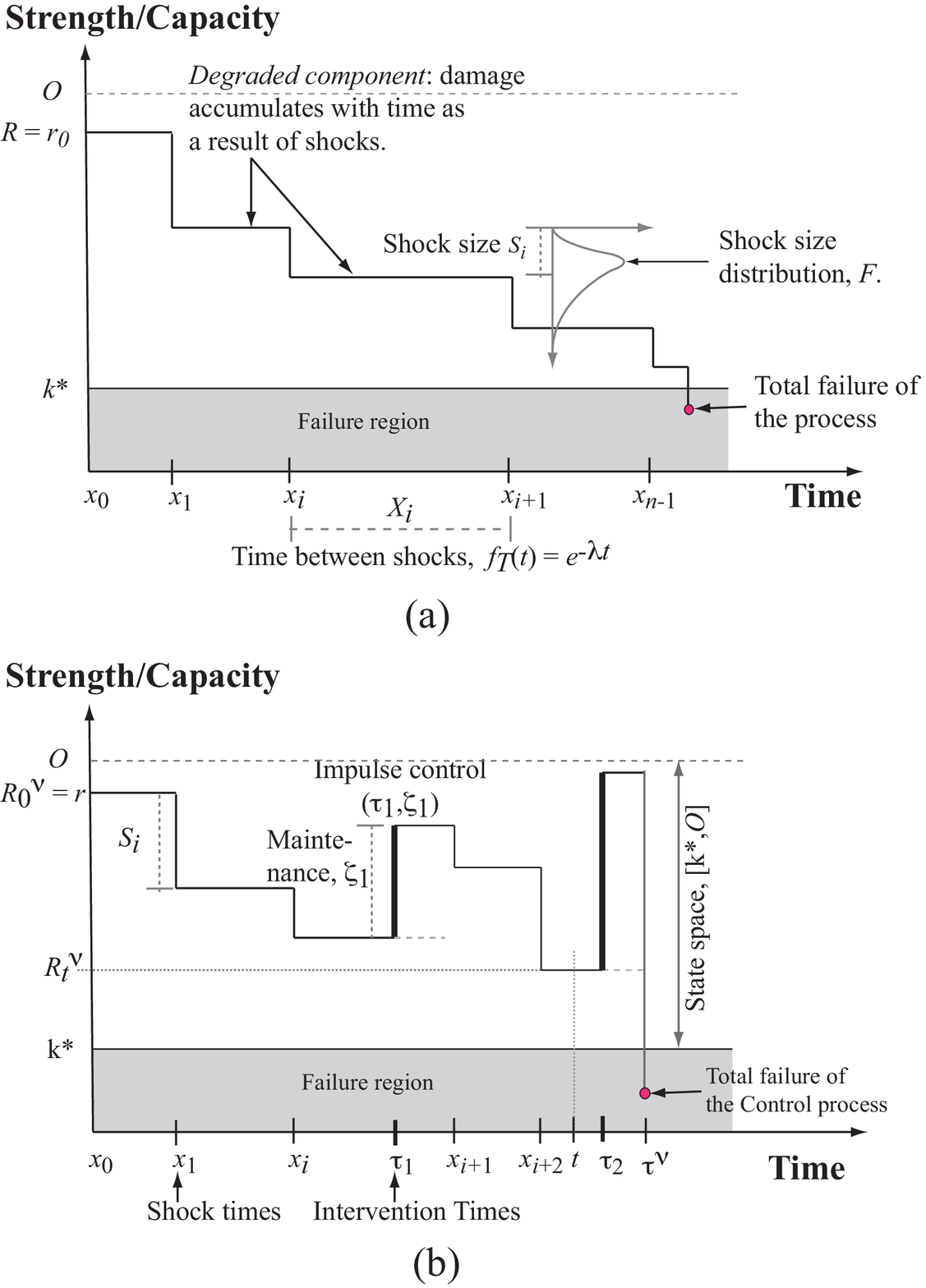}
\caption{Description of the impulse control model.}
\label{Fig1}
\end{center}
\end{figure}

Total failure occurs when the system performance falls bellow a pre-defined threshold $k^*$ with $0 \leq k^* $. Thus, the time of total failure of the controlled process is denoted by 

\begin{equation}
\tau^{\nu}=\inf\{t>0|R^{\nu}_t\leq k^*\},
\end{equation}

\noindent and it is assumed that if the system reaches the threshold the process is stopped. We denote by $\tau$ the time of total failure of the uncontrolled process $R$. Without any loss of generality we assume that $k^*=0$. While $k^*$ denote a lower limit for the process, we also assume that there is a maximum (i.e., optimum) performance level $O$ that cannot be improved. 

Any intervention (i.e., maintenance) at time $\tau_i$ depends on the state of the system just before the intervention $R_{\tau_i-}$. Therefore, the set of possible actions is $[0,O-R_{\tau_i-}]$ and we call $\mathcal{I}$ the set of impulse controls such that $\zeta_i\in[0,O-R_{\tau_i-}]$ for all $i$. We will only consider maintenance policies that are impulse controls and maintenance policies that are impulse controls in $\mathcal{I}$.

If we denote $\mathbb{E}_{r_0}[\cdot]:=\mathbb{E}[\cdot |R^{\nu}_{0-}=r_0]$, for a given $\nu\in\mathcal{I}$ and initial component state $r_0\in[0,O]$, the expected profit (Benefits-Costs) can be computed as:

\begin{equation}
J(r_0,\nu)=\mathbb{E}_{r_0}\left[\int_0^{\tau^{\nu}}e^{-\delta s}G(R_s^{\nu})ds-\sum\limits_{\tau_i<\tau^{\nu}}e^{-\delta\tau_i}C(R^{\nu}_{\tau_i-},\zeta_i)\right],
\label{CBeq}
\end{equation}

\noindent where $G$ is a non-negative continuous, increasing and concave function on $[0,O]$ with $G(0)=0$. $C$ is continuous and increasing in both variables function such that $C>0$, and $\delta$ is the discount factor. The term $e^{-\delta s}$ corresponds to the discounting function used to evaluate the net present value. Note that the first term in equation \eqref{CBeq} corresponds to the discounted benefits; where the function $G$ can be interpreted as an utility function. On the other hand, the second term describes the discounted costs of interventions with $C(r,\zeta)$ the cost of bringing the system from level $r$ to level $r+\zeta$.

The objective of the model is then to find the policy that maximizes the profit among all admissible impulse controls; in other words, we want to find 

\begin{equation}
V(r_0)=\sup_{\nu\in\mathcal{I}}J(r_0,\nu)
\label{Vobj}
\end{equation}

\noindent for a given level $r_0$ in the state space $[0,O]$. Note that it is very difficult to calculate $V(r_0)$ directly from \eqref{Vobj}. First, given a policy $\nu$, we can use Monte Carlo simulations to estimate the expected profit $J(r_0,\nu)$. Then, we have to repeat this process for all possible policies $\nu$, which clearly cannot be obtained at a reasonable computational cost.

Instead, we will solve the problem for all $r\in[0,O]$ at the same time, that is, we want to find the \textit{value function}

\begin{equation}
V(r)=\sup_{\nu\in\mathcal{I}}J(r,\nu).
\label{ObjFunc}
\end{equation}

Although apparently this is a harder problem, we will characterize $V$ as the unique solution of certain equation (that does not involve expectation) and solve this equation numerically. 

From the definition of the value function, we can easy see that $V\geq0$ since we can choose to do nothing. Also, $V(0)=0$ and $V$ is bounded. We will use this to characterize the function $V$.

\section{Preliminaries}

In this section we will present some definitions and fundamental concepts that are required to find $V$ in equation \eqref{ObjFunc}. We start with the following lemma whose proof is presented in the Appendix.

\begin{lemma}\label{lema}
Let $T$ be a stopping time with respect to the filtration $\mathcal{F}_t$.  Then for all $r\in[0,O]$

\begin{equation}\label{dynamic}
V(r)\geq\mathbb{E}_r\left[\int_0^{T\wedge\tau}e^{-\delta s}G(R_s)ds+e^{-\delta T}V(R_T)I_{\{T<\tau\}}\right].
\end{equation}

Furthermore, we have equality in \eqref{dynamic} if it is not optimal to intervene the system before $T$.
\end{lemma}

In order to characterize the value function $V$ in equation \eqref{ObjFunc} we need to define two operators. The first one is the \textit{intervention operator} $\mathcal{M}$ defined as

\begin{equation}\label{interoper}
\mathcal{M}f(r)=\sup_{0\leq\zeta\leq O-r}f(r+\zeta)-C(r,\zeta)
\end{equation}

\noindent for a given function $f$ defined on $[0,O]$ and $r$ in the same interval. We are interested in applying $\mathcal{M}$ to the function $V$. Hence, if we consider any policy $\nu$ such that $\tau_1=0$ and write $\nu=(0,\zeta)\cup\hat{\nu}=(0,\zeta)\cup\{(\tau_i,\zeta_i)\}_{i\geq2}$, then

$$V(r)\geq J(r,\nu)=-C(r,\zeta)+J(r+\zeta,\hat{\nu}),$$

\noindent and since $\hat{\nu}$ is arbitrary we obtain

$$V(r)\geq V(r+\zeta)-C(r,\zeta).$$

Taking the supremum over all admissible $\zeta$, we obtain that

\begin{equation}\label{interV}
V(r)\geq\sup_{0\leq\zeta\leq O-r}V(r+\zeta)-C(r,\zeta)=\mathcal{M}V(r).
\end{equation}

The second operator is the \textit{infinitesimal generator} of the uncontrolled process $R$, that is:

\begin{equation}\label{infinoper}
\mathcal{A}f(r)=\lambda\left(\int_0^rf(r-s)dF(s)-f(r)\right).
\end{equation}

The infinitesimal operator has the property that the process
\begin{equation}\label{martingale}
e^{-\delta t}f(R_t)-f(r)+\int_0^te^{-\delta s}\left(\delta f(R_s)-\mathcal{A}f(R_s)\right)ds
\end{equation}

is a Martingale with respect to  $\mathcal{F}_t$, for bounded $f$ (see \cite{rogerswilliams} and \cite{TAlb}). Taking expectation in \eqref{martingale} and using Optional Sampling Theorem we obtain the so-called Dynkin's Formula: given $T_1\leq T_2$ almost sure (a.s.) finite stopping times, then

\begin{align}\label{Dynk}
\mathbb{E}[e^{-\delta T_2}&f(R_{T_2})-e^{-\delta T_1}f(R_{T_1})]\nonumber\\
&=\mathbb{E}\left[\int_{T_1}^{T_2}e^{-\delta s}\left(\mathcal{A}f(R_s)-\delta f(R_s)\right)ds\right].
\end{align}
Bellow, we will use this Formula with $f$ replaced by $V$.

\section{Optimal maintenance policy}

Since the process $R$ is Markovian, the future is independent of the past given the present. Thus, in order to obtain an optimal policy it is necessary to differentiate between the component states at which an intervention is required and those where there is no need to intervene the system. It is important to stress that because of the Markovian property, this classification will always be the same and will only depend on the state of the system.

We use now the intervention operator $\mathcal{M}$ to describe the optimal policy. Using equation \eqref{interV} $V\geq\mathcal{M}V$,  we can divide the state space $[0,O]$ into the subsets

$$A=\{r\in[0,O]:V(r)=\mathcal{M}V(r)\}$$
and
$$B=\{r\in[0,O]:V(r)>\mathcal{M}V(r)\}.$$

For $r\in A$ we must intervene the system immediately and improve the performance process by $\zeta^*$, where

\begin{eqnarray}
\mathcal{M}V(r)&=&V(r+\zeta^*)-C(r,\zeta^*)\notag\\&=&\sup_{0\leq\zeta\leq O-r}V(r+\zeta)-C(r,\zeta).
\end{eqnarray}

Therefore, we call the set $A$ the {\it intervention region}.  

Now, for $r\in B$ we must do nothing and let the system evolve. Therefore, we obtain equality in \eqref{dynamic}, and using Dynkin's Formula we have that $\delta V(r)-\mathcal{A}V(r)=G(r)$. We call the set $B$ the {\it no intervention region} (Fig. \ref{Fig2}).

\begin{figure}[htbp] 
\begin{center}
\includegraphics[scale=0.6]{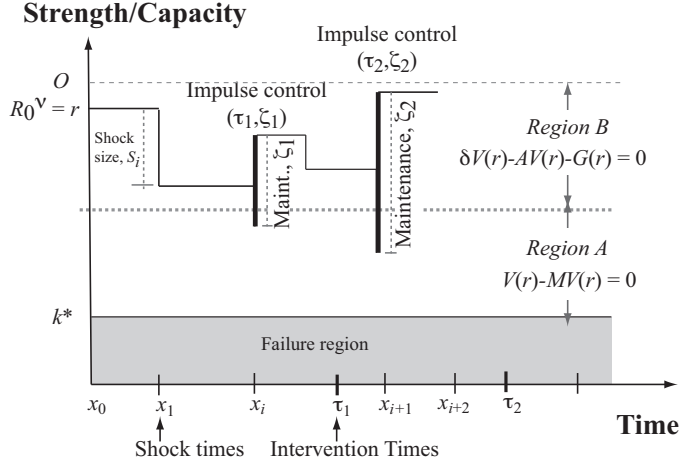}
\caption{Description of the impulse control model.}
\label{Fig2}
\end{center}
\end{figure}

The previous discussion is the intuition behind the following theorem (proved in Appendix).

\begin{theorem}\label{propqvi}
The value function $V$ solves the equation
\begin{equation}\label{qvi}
\min\{\delta V(r)-\mathcal{A}V(r)-G(r),V(r)-\mathcal{M}V(r)\}=0,
\end{equation}
for all $r\in[0,O]$.
\end{theorem}

\paragraph*{Remark} It is possible that $\zeta^*$ is not attainable in the equation above. In this case there is not attainable optimum policy, but we can find controls with expected profit within any degree of accuracy from the value function $V$.

To obtain a full characterization of the value function $V$ (and the optimal policy), we show now that $V$ is the only solution of \eqref{qvi}.
\begin{theorem}
Let $g$ be a non-negative bounded function on $[0,O]$ that solves \eqref{qvi} such that $g(0)=0$. Then $g=V$.
\end{theorem}
\begin{proof}
Let $\nu=\{(\tau_i,\zeta_i)\}_{i\in\mathbb{N}}\in\mathcal{I}$ and initial reliability level $r\in[0,O]$. Using Dynkin's Formula, for $t\geq0$ we have
\begin{align*}
\mathbb{E}_r[&e^{-\delta(t\wedge\tau^{\nu})}g(R^{\nu}_{t\wedge\tau^{\nu}})]\\
=&g(r)+\mathbb{E}_r\left[\int_0^{t\wedge\tau^{\nu}}e^{-\delta s}\left(-\delta g(R^{\nu}_s)+\mathcal{A}g(R^{\nu}_s)\right)ds\right]\\
&+\mathbb{E}_r\left[\sum\limits_{\tau_i<t\wedge\tau^{\nu}}e^{-\delta\tau_i}(g(R^{\nu}_{\tau_i-}+\zeta_i)-g(R^{\nu}_{\tau_i-}))\right].
\end{align*}
Since $g$ solves \eqref{qvi}, then 
\begin{align*}
\mathbb{E}_r[e^{-\delta(t\wedge\tau^{\nu})}g(R^{\nu}_{t\wedge\tau^{\nu}}&)]\\
\leq&g(r)-\mathbb{E}_r\left[\int_0^{t\wedge\tau^{\nu}}e^{-\delta s}G(R^{\nu}_s)ds\right]\\
&+\mathbb{E}_r\left[\sum\limits_{\tau_i<t\wedge\tau^{\nu}}e^{-\delta\tau_i}C(R^{\nu}_{\tau_i-},\zeta_i)\right].
\end{align*}
Letting $t\rightarrow\infty$, bounded convergence we get $J(r,\nu)\leq g(r)$
and taking $\sup$ over $\mathcal{I}$ we obtain 
$$V(r)\leq g(r).$$ 
To prove the reverse inequality, let $\epsilon>0$ and define the following admissible impulse control $\nu^g$: 
$$\tau_i^g=\inf\{t\geq\tau_{i-1}^g:g(R_t^{\nu^g})=\mathcal{M}g(R_t^{\nu^g})\}$$
and $\zeta_i^g$ such that 
$$g(R^{\nu^g}_{\tau_i^g-})\leq g(R^{\nu^g}_{\tau_i^g-}+\zeta_i^g)-C(R^{\nu^g}_{\tau_i^g-},\zeta_i^g)+\frac{\epsilon}{2^i}.$$
Note that $\delta g(R^{\nu^g}_s)-\mathcal{A}g(R^{\nu^g}_s)=G(R^{\nu^g}_s)$ for $s$ between interventions. Therefore, from the previous calculations we obtain
\begin{align*}
\mathbb{E}_r[e^{-\delta(t\wedge\tau^{\nu^g})}&g(R^{\nu^g}_{t\wedge\tau^{\nu^g}})]\\
\geq&g(r)-\mathbb{E}_r\left[\int_0^{t\wedge\tau^{\nu^g}}e^{-\delta s}G(R^{\nu^g}_s)ds\right]\\
&+\mathbb{E}_r\left[\sum\limits_{\tau^g_i<t\wedge\tau^{\nu^g}}e^{-\delta\tau^g_i}C(R^{\nu^g}_{\tau^g_i-},\zeta^g_i)\right]-\epsilon.
\end{align*}
Letting $t\rightarrow\infty$ again, we get $g(r)\leq J(\nu^g,r)+\epsilon\leq V(r)+\epsilon$. Since $\epsilon$ is arbitrary we get the reverse inequality
$$g(r)\leq V(r).$$
\end{proof}

\section{Numerical solution}
\label{Numsol}

To obtain the optimal policy is necessary to find the value function $V$ by solving equation \eqref{qvi}. Once we have $V$ we can compute the intervention and no intervention regions. Thus, for $r\in B$, we have that 

$$\delta V(r)-\mathcal{A}V(r)-G(r)=0.$$

Using the definition of the infinitesimal operator $\mathcal{A}$ (equation \eqref{infinoper}) with $f=V$ and solving the above equation for $V(r)$ it is obtained that 

\begin{equation}
V(r)=\frac{1}{\lambda+\delta}\left\{G(r)+\lambda \int_0^rV(r-y)dF(y) \right\}.
\end{equation}

On the other hand, for the region where interventions are required; i.e., $r\in A$, we have that

$$V(r)-\mathcal{M}V(r)=0.$$

Then, using the definition of the intervention operator (equation \eqref{interoper}), $V(r)$ satisfies:

\begin{equation}
V(r)=\sup_{0\leq\zeta\leq O-r}V(r+\zeta)-C(r,\zeta).
\end{equation}

To approximate $V$ we follow the Jacobi iteration method described in \cite{Kushner}. First, we discretize the interval $[0,O]$ in $N=\frac{1}{h}$ intervals and initizialize the vector $V^h_0\equiv0\in\mathbb{R}^{N+1}$. Now, given $V^h_n$ we compute
\begin{align*}
V^h&_{n+1}[j]=\\
&\max\left\{\frac{1}{\lambda+\delta}\left\{G(jh)+\lambda\sum\limits_{i=0}^jV^h_n[j-i]p_i \right\},\mathcal{M}V^h_n[j]\right\},
\end{align*}
where $p_i$ is the discretized density of $F$ for $ih$ and $$\mathcal{M}V^h_n[j]=\max\limits_{0\leq i\leq N+1-j}V^h_n[j+i]-C(jh,ih).$$ Note that $V_n^h[0]=0$ for all $n$ and all $h$ since $V(0)=0$.

We continue iterating over $n$ until 

$$\max\limits_j\left|V^h_n[j]-V^h_{n+1}[j]\right|<\epsilon,$$
for a given error tolerance $\epsilon$.

\section{Illustrative example}

Consider an infrastructure system subject to shocks (e.g., earthquakes) whose occurrence times follow an exponential distribution with rate $\lambda = 0.5$. The sizes of the shocks are log-normally distributed with $\mu=0.3$ and $\sigma=1$; and, consequently, distribution parameters $\mu_{log}=2.23$ and $\sigma_{log}=2.9$. Furthermore, assume that the performance (state) of the system is permanently monitored. This means that it is possible to know the state of the system when required (e.g., at periodic inspection times). The system state (which should be in practice measured in physical units) is normalized and evaluated within the interval $[0,1]$; where 1 means that the system is in ``as good as new'' condition, and $k^*=0$ indicates that it is not in operating condition. The objective of the study is to define an optimal maintenance policy.

For the purpose of this example, the following assumptions are made. The function $G(r)$ (equation \eqref{CBeq}), which is the utility function is given by:

\begin{equation}
G(r)=C\frac{1}{\alpha}(1-e^{-\alpha r}),
\end{equation} 

\noindent where $C=5$ and $\alpha = 2$. Note that this curve has the form of an exponential risk aversion utility function. On the other hand, it is assumed that the costs associated to an intervention are given by the following function (equation \eqref{CBeq}),

\begin{equation}
C(r,\zeta)=r+\zeta^2+K
\end{equation}

\noindent where the constant $K=0.1$ reflects the fixed costs of any intervention. Note that the intervention costs are proportional to the current state of the system and grow with the square of the size of the intervention. For both utility and cost, these values are discounted to the time of the decision by using a discount factor $\delta = 0.2$.

The outcome of the approach consists of two parts. First, it is necessary to define, for every system state $r$, the intervention intensity $\zeta$ that maximizes the expected profit (equation \ref{CBeq}). This requires dividing the system in two states: a region (state space values) where no intervention is required and a set of values for which it is necessary. Thus, for a given system state obtained (measured) at the time of inspection, the intervention level that maximizes the expected profit at that particular time (i.e., discounted) can be obtained from this result. The second result, is the value function $V$ that provides the optimum expected profit value obtained if the required actions (defined previously) are conducted.

\begin{figure}[htbp] 
\begin{center}
\includegraphics[scale=0.6]{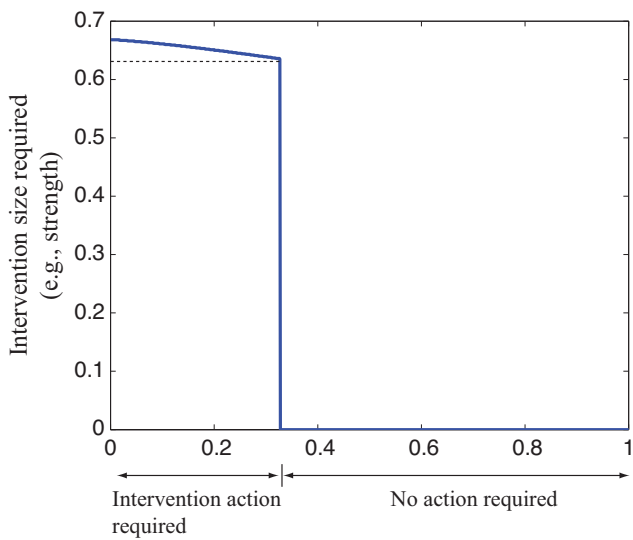}
\caption{Description of the intervention requirements}
\label{Fig3}
\end{center}
\end{figure}

Following the numerical approach presented in section \ref{Numsol} we obtain the results after 8 iterations at a minimum computational time. The division between the intervention and not intervention states (i.e. regions $A$ and $B$) can be observed in Fig. \ref{Fig3}. Clearly, if the system is operating at a level $r>0.328$ there is no need for an intervention. However, if an inspection indicates that its state is $r\leq 0.328$ and intervention is required and the size of the intervention is shown in the figure. For instance, if after an inspection at time $t$ the system state is $R_t=0.6$ no intervention is required, but if $R_t=0.3$ and intervention of magnitude $\zeta=0.64$ will be necessary to maximize the profit. This means that immediately after the intervention the system will be at state 0.94 (i.e., 0.3+0.64=0.94). Notice that Fig. \ref{Fig3} does not change with time as mentioned earlier. If in the next inspection the observed state of the system is, say 0.6, no intervention will be necessary

Furthermore, if this maintenance policy is carried out, the maximum expected profit can be observed in Fig. \ref{Fig4}. Therefore, if after an inspection at time $t$ the system state is $R_t=0.3$, the maximum expected profit that we can obtain, by following the above policy, is $V(0.3)=3.77$. Otherwise, if the inspection yields a state $R_t=0.6$, we can obtain a maximum of $V(0.6)=4.13$.

\begin{figure}[htbp] 
\begin{center}
\includegraphics[scale=0.6]{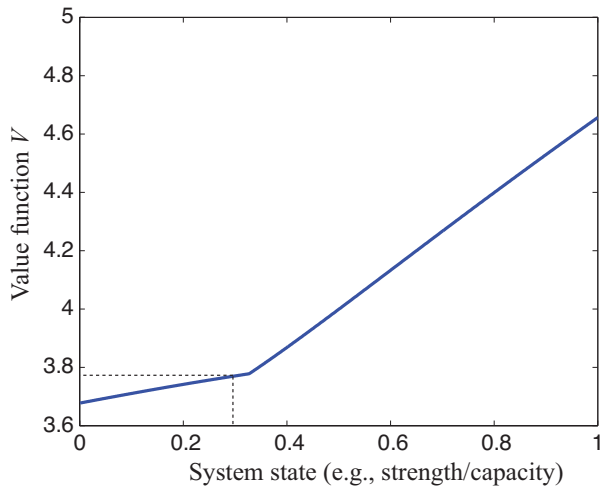}
\caption{Maximum expected profit.}
\label{Fig4}
\end{center}
\end{figure}

\section{Summary and conclusions}

The paper presents an approach to define the optimum maintenance policy of a system that deteriorates with time as a result of shocks. The deterioration process is modeled as a compound Poisson process. The proposed maintenance strategy follows an impulse control model that requires the permanent (or at least frequent) monitoring of the system state. Then, at every inspection time the model can be used to make a decision as to whether the system should be intervened or not. In case of requiring an intervention, the extent of the optimum repair can be obtained from the model. The decisions based on the model guarantee that the net present value of the utility at the time of the intervention is maximum. It is suggested in the paper that this maintenance approach is of particular importance for engineering systems that operate under adverse environments for long time periods (e.g., $>25$ years); for instance, physical infrastructure (e.g., bridges, highways). Traditional maintenance strategies define long term inspection and maintenance plans at a given point in time (usually $t=0$) without considering possible variations in the system's use or condition, and the technology available at the time of the decisions. Thus, it is argued that the best maintenance policy in these cases can be obtained by combining both permanent monitoring and optimum interventions. 

\appendix
\setcounter{equation}{6}
\begin{lema1}
Let $T$ be a stopping time with respect to the filtration $\mathcal{F}_t$.  Then for all $r\in[0,O]$

\begin{equation}
V(r)\geq\mathbb{E}_r\left[\int_0^{T\wedge\tau}e^{-\delta s}G(R_s)ds+e^{-\delta T}V(R_T)I_{\{T<\tau\}}\right].
\end{equation}

Furthermore, we have equality in \eqref{dynamic} if it is not optimal to intervene the system before $T$.
\end{lema1}

\begin{proof}
Let $r\in[0,O]$ and $\nu=\{(\tau_i,\zeta_i)\}_{i\in\mathbb{N}}\in\mathcal{I}$, then $J(r,\nu)=\mathbb{E}_r\left[\eta\right]$ where
$$\eta=\int_0^{\tau^{\nu}}e^{-\delta s}G(R_s^{\nu})ds-\sum\limits_{\tau_i<\tau^{\nu}}e^{-\delta\tau_i}C(R^{\nu}_{\tau_i-},\zeta_i).$$
Given $T$ a stopping time with respect to the filtration $\mathcal{F}_t$, the control $\nu+T=\{(\tau_i+T,\zeta_i)\}_{i\in\mathbb{N}}$ is also an admissible control. If we call $\Theta_t$ the usual shift operator (see \cite{rogerswilliams}), from the strong Markov property of the process $R$ we have
\begin{align*}
J(r,\nu&+T)\\
&=\mathbb{E}_r\left[\int_0^{T\wedge\tau}e^{-\delta s}G(R_s)ds+e^{-\delta T}\Theta_T\eta I_{\{T<\tau\}}\right]\\
&=\mathbb{E}_r\left[\int_0^{T\wedge\tau}e^{-\delta s}G(R_s)ds+e^{-\delta T}\mathbb{E}_{R_T}[\eta]I_{\{T<\tau\}}\right]\\
&=\mathbb{E}_r\left[\int_0^{T\wedge\tau}e^{-\delta s}G(R_s)ds+e^{-\delta T}J(R_T,\nu)I_{\{T<\tau\}}\right]
\end{align*}
Taking $\sup$ over $\nu\in\mathcal{I}$ we obtain \eqref{dynamic}. If it is optimal not to intervene before $T$, then $V(r)=\sup_{\nu\in\mathcal{I}}J(r,\nu+T)$.
\end{proof}

\setcounter{equation}{12}
\begin{teorema1}
The value function $V$ solves the equation
\begin{equation}\label{qvi}
\min\{\delta V(r)-\mathcal{A}V(r)-G(r),V(r)-\mathcal{M}V(r)\}=0,
\end{equation}
for all $r\in[0,O]$.
\end{teorema1}

\begin{proof}
Let $r\in[0,O]$. From equation \eqref{interV} $V\geq\mathcal{M}V$. On the other hand, for any stopping time $T\leq\tau$ a.s., by Dynkin's Formula and boundedness of $V$, we have
\begin{align*}
V(r)-\mathbb{E}_r[e^{-\delta T}&V(R_T)]\\
&=\mathbb{E}_r\left[\int_0^Te^{-\delta s}\left(\delta V(R_s)-\mathcal{A}V(R_s)\right)ds\right].
\end{align*}
Using \eqref{dynamic} we get that
$$\mathbb{E}_r\left[\int_0^Te^{-\delta s}\left(\delta V(R_s)-\mathcal{A}V(R_s)-G(R_s)\right)ds\right]\geq0.$$
If we choose $T=T_1$, the time of the first shock, then $R_s=r$ for $0\leq s< T_1$ and therefore $\delta V(r)-\mathcal{A}V(r)-G(r)\geq0$. Now, suppose that $V(r)>\mathcal{M}V(r)$, hence it is not optimal to intervene at time 0, so it is not optimal to intervene before $T_1$. Then, by lemma \ref{lema} we get $\delta V(r)-\mathcal{A}V(r)-G(r)=0$.
\end{proof}

\nocite{*}
\bibliography{referencias}
\bibliographystyle{abbrv}

\end{document}